\documentclass[12pt]{article}
\usepackage{amssymb, amsmath, amsthm}

\theoremstyle{theorem}
\newtheorem{theorem}{Theorem}[section]
\newtheorem{lemma}[theorem]{Lemma}
\newtheorem{cor}[theorem]{Corollary}
\theoremstyle{definition}

\newtheorem{conjecture}[theorem]{Conjecture}
\DeclareMathOperator{\lcm}{lcm}

\newcommand{\ov}{\overline}
\newcommand{\calc}{{\cal C}}
\newcommand{\dd}{\displaystyle}

\title{\bf A generalization of the Euler's totient function}
\author{Marius T\u arn\u auceanu}
\date{December 5, 2013}

\begin{document}

\maketitle

\begin{abstract}
The main goal of this paper is to provide a group theoretical
gene\-ralization of the well-known Euler's totient function. This
determines an interesting class of finite groups.
\end{abstract}


\bigskip

\section{Introduction}

The {\it Euler's totient function} (or, simply, the {\it totient
function}) $\varphi$ is one of the most famous functions in number
theory. Recall that the totient $\varphi(n)$ of a positive integer
$n$ is defined to be the number of positive integers less than or
equal to $n$ that are coprime to $n$. The totient function is
important mainly because it gives the order of the group of all
units in the ring ($\mathbb{Z}_n$, +, $\cdot$). Also, $\varphi(n)$
can be seen as the number of generators of the finite cyclic group
($\mathbb{Z}_n$, +).

Many generalizations of the totient function are known (for
example, see \cite{5}, \cite{6}, \cite{13} and the special chapter
on this topic in \cite{11}). From these, the most significant is
probably the {\it Jordan's totient function} (see \cite{4}).

In this paper we will introduce and study a new generalization of
$\varphi$ that uses group theory ingredients.

A basic result on finite groups states that the order $o(\hat{a})$
of an element $\hat{a}\in \mathbb{Z}_n$ is given by the formula
$$o(\hat{a})=\frac{n}{\gcd(a,n)}\hspace{0,5mm}.$$It shows that $\varphi(n)$
is in fact the number of elements of order $n$ in $\mathbb{Z}_n$,
or equivalently
$$\varphi(n)=\hspace{1mm}\mid\hspace{-1mm}\{\hat{a}\in
\mathbb{Z}_n \mid o(\hat{a})={\rm
exp}(\mathbb{Z}_n)\}\hspace{-1mm}\mid\hspace{0,5mm},$$where ${\rm
exp}(\mathbb{Z}_n)$ denotes the exponent of $\mathbb{Z}_n$. This
expression of $\varphi(n)$ in which only group theoretical notions
are involved constitutes the starting point for our discussion. It
can naturally be extended to an arbitrary finite group $G$, by
putting
$$\varphi(G)=\hspace{1mm}\mid\hspace{-1mm}\{a\in
G \mid o(a)=\exp(G)\}\hspace{-1mm}\mid.$$Since
$\varphi(\mathbb{Z}_n)=\varphi(n)$, for all $n\in\mathbb{N}^*$, a
generalization of the classical totient function $\varphi$ is
obtained. Notice that the above $\varphi$ is not a function (more
exactly, by endowing $\mathbb{N}$ with a category structure and
defining a suitable action of $\varphi$ on group homomorphisms, it
can be seen as a functor from the category of finite groups to
this category).

The paper is organized as follows. Some basic properties and
results on $\varphi$ are presented in Section 2. In Section 3 we
study the connections of $\varphi(G)$ with
$\varphi(\mid\hspace{-1mm} G\hspace{-1mm}\mid)$ and
$\mid\hspace{-1mm}{\rm Aut}(G)\hspace{-1mm}\mid$\,. Section 4
deals with the class of finite groups $G$ for which
$\varphi(G)\neq 0$. In the final section several conclusions and
further research directions are indicated.

Most of our notation is standard and will not be repeated here.
Basic definitions and results on group theory can be found in
\cite{8} and \cite{14}. For number theoretic notions we refer the
reader to \cite{4}, \cite{10} and \cite{11}.
\bigskip

\section{Basic properties of $\varphi$}

First of all, we study some basic properties of $\varphi$ derived
from similar pro\-perties of the classical totient function.

Clearly, $\varphi$ preserves isomorphisms, that is the group
isomorphism $G_1 \cong G_2$ implies that
$\varphi(G_1)=\varphi(G_2)$. Also, for
any finite cyclic groups $G$, we have
$\varphi(G)=\varphi(\mid\hspace{-1mm} G \hspace{-1mm}\mid)$.
Then $\varphi(\mathbb{Z}_3)=
\varphi(\mathbb{Z}_4)$, but the groups $\mathbb{Z}_3$ and
$\mathbb{Z}_4$ are not isomorphic, a property which corresponds to
the non-injectivity of the totient function.

Regarding the values of $\varphi$, we observe that for a finite
group $G$ with $\exp(G)= m$, we have $\varphi(G)=\varphi(m)k$,
where $k$ is the number of cyclic subgroups of order $m$ in $G$.
It is well-known that $\varphi(m)$ is even for all $m\geq3$. On
the other hand, if $m=2$, then $G$ is an elementary abelian
2-group, say $G\cong\mathbb{Z}_2^n$, and we easily get $k=2^n-1$.
Consequently, the only odd numbers contained in $Im(\varphi)$ are
of the form $2^n-1$, $n\in \mathbb{N}^*$.
\bigskip

Another basic property of the totient function is the following:
$$m \mid n \hspace{1mm}\Longrightarrow\hspace{1mm} \varphi(m) \mid \varphi(n) \mbox{ and }
\varphi(m)\leq \varphi(n).$$This implication fails for the
generalized $\varphi$. Indeed, by taking a subgroup $K \cong
\mathbb{Z}_2 \times \mathbb{Z}_2$ of the dihedral group $D_8$, one
obtains $\varphi(K)=3>2=\varphi(D_8)$.
\bigskip

Next we will focus on computing $\varphi(G)$ for some remarkable
classes of finite groups $G$. In several cases the following lemma
(corresponding to a well-known result on the totient function)
will be very useful.

\begin{lemma}
    $\varphi$ is multiplicative, that is
    if $(G_i)_{i=\ov{1,k}}$ is a family of finite groups of coprime
    orders, then we have:
    \[ \varphi(\prod_{i=1}^kG_i)=\prod_{i=1}^k\varphi(G_i). \]
\end{lemma}
\begin{proof}
Every element $a\hspace{-0,5mm}\in\hspace{-0,5mm}
\prod_{i=1}^kG_i$ can uniquely be written as $a{=}(a_1,a_2,...,
a_k)$, where $a_i \in G_i, i=\ov{1,k}$. Under our hypothesis, we
have
$$o(a)=\prod_{i=1}^ko(a_i) \hspace{1mm}\mbox{ and }\hspace{1mm} \exp(\prod_{i=1}^kG_i)=\prod_{i=1}^k{\rm
exp}(G_i).$$One easily obtains that $o(a)={\rm
exp}(\prod_{i=1}^kG_i)$ if and only if $o(a_i)=\exp(G_i)$, for all
$i=\ov{1,k}$. This shows that there is a bijection between the set
of elements of order ${\rm exp}(\prod_{i=1}^kG_i)$ in
$\prod_{i=1}^kG_i$ and the cartesian product of the sets $\{a_i
\in G_i \mid o(a_i)={\rm exp}(G_i)\}$, $i=1,2, \dots, k$. Hence
the desired equality holds.
\end{proof}

The following theorem shows that the computation of $\varphi(G)$
for a finite nilpotent group $G$ is reduced to $p$-groups.

\begin{theorem}
Let $G$ be a finite nilpotent group and $G_i$,
$i=\ov{1,k}$, be the Sylow subgroups of $G$. Then
$$\varphi(G)=\prod_{i=1}^k\varphi(G_i).$$
\end{theorem}

\begin{proof}
The equality follows immediately from Lemma
2.1, since a finite nilpotent group is the direct product of its
Sylow subgroups.
\end{proof}

If $G$ is a finite $p$-group of order $p^n$, then we have $${\rm
exp}(G)={\rm max}\{o(a) \mid a \in G\}=p^m, \mbox{ where } 0 \leq
m \leq n,$$and $\varphi(G) \geq \varphi(p^m)$ (notice that this
can be an equality, as for the dihedral groups $D_{2^n}$,
$n\geq3$, the generalized quaternion groups $Q_{2^n}$, $n\geq4$,
or the quasi-dihedral groups $S_{2^n}$, $n\geq4$ -- see Theorem
4.1 of \cite{14}, II). We also infer that, in this case, $\varphi(G)\neq 0$.
Unfortunately, an explicit formula for $\varphi(G)$ cannot be
obtained in the general case.
\bigskip

A particular class of finite $p$-groups for which we are able to
compute explicitly $\varphi(G)$ is that of abelian $p$-groups.

\begin{theorem}
Let $G$ be a finite abelian $p$-group of
type $(p^{\alpha_1}, p^{\alpha_2}, \dots, p^{\alpha_r})$ and assume
that $\alpha_1 \leq \alpha_2 \leq \dots \leq
\alpha_{s-1}<\alpha_s=\alpha_{s+1}= \dots =\alpha_r$. Then
$$\varphi(G)=\mid G \mid(1-\frac{1}{p^{\hspace{0,5mm}r-s+1}})\,.$$
\end{theorem}

\begin{proof}
We have $\exp(G)=p^{\alpha_r}$ and therefore $\varphi(G)$ is equal
to the number of elements of order $p^{\alpha_r}$ in $G$. This
number can be easily found by using Corollary 4.4 of \cite{15} for
$\alpha=\alpha_r$. Under the notation of \cite{15}, one obtains
$$\hspace{-20mm}\varphi(G)=p^{\alpha_r}h_p^{r-1}(\alpha_r)-p^{\alpha_r-1}h_p^{r-1}(\alpha_r-1)$$
$$\hspace{7,5mm}=p^{\alpha_1+\alpha_2+\dots+\alpha_r}-p^{\alpha_1+\alpha_2+\dots+\alpha_{s-1}+(r-s+1)(\alpha_r-1)}$$
$$\hspace{-23mm}=p^{\alpha_1+\alpha_2+\dots+\alpha_r}(1-\frac{1}{p^{\hspace{0,5mm}r-s+1}})$$
$$\hspace{-35,5mm}=\,\mid G \mid(1-\frac{1}{p^{\hspace{0,5mm}r-s+1}}),$$completing the proof.
\end{proof}

Obviously, Lemma 2.1 allows us to extend the above result to
arbitrary finite abelian groups.

\begin{cor}
Let $G=\dd\prod_{i=1}^k G_i$ be a finite abelian group, where
$G_i$ is of type $(p_i^{\alpha_{i1}}, p_i^{\alpha_{i2}}, \dots,
p_i^{\alpha_{ir_{i}}})$, and assume that $\alpha_{i1} \leq
\alpha_{i2} \leq \dots \leq
\alpha_{is_{i}-1}<\alpha_{is_{i}}=\alpha_{is_i+1}= \dots
=\alpha_{ir_i}$, $i=\ov{1,k}$. Then
$$\varphi(G)=\prod_{i=1}^k\varphi(G_i)=\prod_{i=1}^k \mid G_i \mid(1-\frac{1}{p_i^{\hspace{0,5mm}r_i-s_i+1}})$$
$$\hspace{-2mm}=\varphi(\mid G \mid)\prod_{i=1}^k\frac{p_i^{\hspace{0,5mm}r_i-s_i+1}-1}{p_i^{\hspace{0,5mm}r_i-s_i+1}-p_i^{\hspace{0,5mm}r_i-s_i}}\,.$$
\end{cor}

An important class of finite (nilpotent) groups whose structure is
strongly connected to abelian groups is that of finite
hamiltonian groups. Such a group $G$ is the direct product of a
quaternion group of order 8, an elementary abelian 2-group and a
finite abelian group $A$ of odd order. The value $\varphi(G)$ can
also be calculated, according to our above results.

\begin{cor}
Let $G=Q_8 \times \mathbb{Z}_2^n \times A$ be a finite hamiltonian group.
Then $$\varphi(G)=3\cdot2^{n+1}\varphi(A),$$where $\varphi(A)$ is
given by Corollary {\rm 2.4}.
\end{cor}

\begin{proof}
By Lemma 2.1, we get
$$\varphi(G)=\varphi(Q_8 \times \mathbb{Z}_2^n)\varphi(A).$$On the
other hand, it is easy to see that $\exp(Q_8 \times
\mathbb{Z}_2^n)=\exp(Q_8)=4$. An element $(a,b)\in Q_8 \times
\mathbb{Z}_2^n$ has order 4 if and only if $o(a)=4$ in $Q_8$.
Since $Q_8$ possesses six elements of order 4, it results that
$\varphi(Q_8 \times \mathbb{Z}_2^n)=6\cdot2^n=3\cdot2^{n+1}$,
which leads to the desired formula.
\end{proof}

The computation of $\varphi(G)$ can also be made for several
classes of finite groups that are not necessarily nilpotent. Two
simple examples of such groups are the finite dihedral groups
$D_{2n}$, $n \geq 2$, and the finite nonabelian $P$-groups (recall
that, given an integer $n\geq2$ and two primes $p>2, q$ such that
$q \mid p-1$, a nonabelian $P$-group $G$ of order $p^{n-1}q$ is a
semidirect product of a normal subgroup $A\cong\mathbb{Z}_p^{n-1}$
by a cyclic subgroup of order $q$ which induces a nontrivial power
automorphism on $A$; moreover, it is well-known that $G$ is
lattice-isomorphic to $\mathbb{Z}_p^n$ -- see Theorem 2.2.3 of
\cite{12}).

\begin{theorem}
The following equalities hold:
\begin{itemize}
\item[\rm a)]
        $\varphi(D_{2n})=\left\{\begin{array}{lll}
        0,&n \equiv 1 \hspace{1mm}({\rm mod}\hspace{1mm} 2)\\
        &&\mbox{, for all }n\geq 3.\\
        \varphi(n),&n \equiv 0 \hspace{1mm}({\rm mod}\hspace{1mm} 2)\end{array}\right.$
\item[\rm b)]
        $\varphi(G)=0$, where $G$ is the nonabelian $P$-group of order $p^{n-1}q$ {\rm(}$p>2,q$ primes, $q \mid p-1${\rm)} that is $L$-isomorphic with $\mathbb{Z}_p^n$.
\end{itemize}
\end{theorem}

\begin{proof}
a) The dihedral group $D_{2n}= \langle x,y\mid x^n=y^2=1,\
yxy=x^{-1} \rangle$, $n\ge2$, has a unique cyclic (normal)
subgroup of order $n$, namely $\langle x \rangle$, and all
elements in $D_{2n}\setminus\langle x \rangle$ are of order 2. We
infer that
$$\exp(D_{2n})=\left\{\begin{array}{lll}
        2n,&n \equiv 1 \hspace{1mm}({\rm mod}\hspace{1mm} 2)\\
        &&\\
        n,&n \equiv 0 \hspace{1mm}({\rm mod}\hspace{1mm}
        2)\end{array}\right.$$and the desired expression for
        $\varphi(D_{2n})$ follows immediately.

b) Since $G$ contains only elements of orders $p$ or $q$, we have
$\exp(G)=pq$ and hence $\varphi(G)=0$.
    \qedhere
\end{proof}

From Theorem 2.3 we obtain $\varphi(\mathbb{Z}_p^n)=p^n-1$. On the
other hand, by b) of Theorem 2.6, for the finite nonabelian
$P$-group $G$ of order $p^{n-1}q$ ($p>2,q$ primes, $q \mid p-1$)
we have $\varphi(G)=0$, even though $G$ is $L$-isomorphic to
$\mathbb{Z}_p^n$. This remark leads to the following result.

\begin{cor}
$\varphi$ does not preserve $L$-isomorphisms.
\end{cor}

The most general finite groups are the symmetric groups $S_n$, $n
\in \mathbb{N}^*$. Obviously, the values $\varphi(S_n)$ can be
easily computed for the first positive integers $n$ (e.g.
$\varphi(S_1)=\varphi(S_2)=1$, $\varphi(S_3)=\varphi(S_4)=0$,
\dots, and so on) and the same thing can be also said about the
values $\varphi(A_n)$, where $A_n$ is the alternating group on $n$
letters (e.g. $\varphi(A_2)=1$, $\varphi(A_3)=2$,
$\varphi(A_4)=\varphi(A_5)=0$, \dots, and so on). For an arbitrary
$n$ the above values are given by the following theorem.

\begin{theorem}
\begin{enumerate}
    \item[\rm a)]
    For all $n\geq3$, we have $\varphi(S_n)=0$.
    \item[\rm b)]
    For all $n\geq4$, we have $\varphi(A_n)=0$.
\end{enumerate}
\end{theorem}

\begin{proof}
a) Since the order of a permutation in $S_n$ is the least common
multiple of the lengths of the cycles in its cycle decomposition,
we get
\begin{equation}\label{eq:expSn}
    \exp(S_n)= \lcm(1,2, \dots, n) = \prod_{p_i \leq n} p_i^{\alpha_i} \,,
\end{equation}
where the product runs over all primes less than $n$, and for each such $p_i \leq n$, the exponent
$\alpha_i$ is the largest number such that $p_i^{\alpha_i} \leq n$, i.e.\@
\begin{equation}\label{eq:piai}
    n / p_i < p_i^{\alpha_i} \leq n \,.
\end{equation}
An element $g \in S_n$ has order equal to $\exp(S_n)$ if and only if it has a cycle decomposition
into non-trivial cycles of lengths $n_1,\dots,n_k$ with $n_1 + \dots + n_k \leq n$ and $\lcm(n_1,\dots,n_k) = \exp(S_n)$.
Since every factor $p_i^{\alpha_i}$ appearing in~\eqref{eq:expSn} has to occur in at least one
element $n_j$, this implies that
\[ \sum_{p_i \leq n} p_i^{\alpha_i} \leq n \,. \]
Together with~\eqref{eq:piai}, this gives
\[ n \sum_{p_i \leq n} (1/p_i) < \sum_{p_i \leq n} p_i^{\alpha_i} \leq n \,. \]
Since $1/2 + 1/3 + 1/5 > 1$, this yields $n \leq 4$. If $n=4$,
then $\exp(S_n) = 12$, but $S_4$ has no element of order $12$. If
$n=3$, then $\exp(S_n) = 6$, but $S_3$ has no element of order $6$
either.

b) In order to compute $\exp(A_n)$, we note that each cycle of odd
length in $S_n$ is contained in $A_n$, and hence the odd parts of
$\exp(A_n)$ and $\exp(S_n)$ coincide. For the $2$-part, observe
that every cycle of $S_n$ of even length less than $n-2$ can be
extended to an element of $A_n$ by multiplying it by a
transposition disjoint from this cycle, and hence the $2$-part of
$\exp(A_n)$ coincides with the $2$-part of $\exp(S_{n-2})$. We
deduce that
\begin{equation}\label{eq:expAn}
    \exp(A_n) = \begin{cases}
        \exp(S_n) / 2, & \  \hspace{1mm}n = 2^\ell \text{ or } n = 2^\ell + 1 \text{ for some } \ell  \\
        \exp(S_n), & \ \text{ otherwise}.
    \end{cases}
\end{equation}
In particular, we can write
\[ \exp(A_n) = \prod_{p_i \leq n} p_i^{\beta_i} \,, \]
with
\[ n / 4 < 2^{\beta_1} \leq n \quad \text{and} \quad n / p_i < p_i^{\beta_i} \leq n, \text{ for all } p_i \geq 3 \,. \]
If there is an element $g \in S_n$ with order equal to
$\exp(A_n)$, then we infer as before that
\[ n / 4 \hspace{1mm}+ \sum_{3 \leq p_i \leq n} (n/p_i) < \sum_{p_i \leq n} p_i^{\beta_i} \leq n \,. \]
Since $1/4 + 1/3 + 1/5 + 1/7 + 1/11 > 1$, this yields $n \leq 10$.
It is now straightforward to check that for $4 \leq n \leq 10$, the group $A_n$ does not contain elements of order $\exp(A_n)$.
\qedhere
\end{proof}

\section{Connections of $\varphi(G)$ with $\varphi(\mid\hspace{-1mm} G\hspace{-1mm}\mid)$ and $\mid\hspace{-1mm}{\rm Aut}(G)\hspace{-1mm}\mid$}

As follows from our previous results, for some classes of finite
groups $G$ the value $\varphi(G)$ depends on the value of the
classical totient function computed for $\mid G \mid$\,. In this
way, the following tasks are natural:
\begin{itemize}
\item[{\rm a)}] given a finite group $G$, compare $\varphi(G)$ with $\varphi(\mid\hspace{-1mm} G\hspace{-1mm}\mid)$;
\item[{\rm b)}] determine the finite groups $G$ satisfying $\varphi(G)=\varphi(\mid\hspace{-1mm} G \hspace{-1mm}\mid)$.
\end{itemize}

Related to {\rm a)} we are able to indicate three simple examples,
which show that for every relation $R\in
\{<,\hspace{1mm}=,\hspace{1mm}>\}$ there exist finite non-abelian
groups $G$ with $\varphi(G)\hspace{1mm} R
\hspace{1mm}\varphi(\mid\hspace{-1mm} G \hspace{-1mm}\mid)$:
\begin{description}
\item[\hspace{20mm}${\rm a}_1)$] $\varphi(D_8)=2<4=\varphi(\mid\hspace{-1mm} D_8\hspace{-1mm}\mid)$;
\item[\hspace{20mm}${\rm a}_2)$] $\varphi(\mathbb{Z}_3\times S_3)=6=\varphi(\mid\hspace{-1mm} \mathbb{Z}_3\times S_3\hspace{-1mm}\mid)$;
\item[\hspace{20mm}${\rm a}_3)$] $\varphi(Q_8)=6>4=\varphi(\mid\hspace{-1mm}Q_8\hspace{-1mm}\mid)$.
\end{description}

More can be said in the case of abelian groups, for which
Corollary 2.4 easily leads to the following theorem.

\begin{theorem}
Let $G=\dd\prod_{i=1}^k G_i$ be a finite abelian group, where
$G_i$ is of type $(p_i^{\alpha_{i1}}, p_i^{\alpha_{i2}}, \dots,
p_i^{\alpha_{ir_{i}}})$, $i=\ov{1,k}$. Then
$$\varphi(G)\geq\varphi(\mid\hspace{-1mm}G\hspace{-1mm}\mid),$$and
we have equality if and only if $\alpha_{ir_i-1}<\alpha_{ir_i}$,
for all $i=\ov{1,k}$, that is if and only if $G$ has a unique
cyclic subgroup of order $\exp(G)$.
\end{theorem}

For an arbitrary finite group $G$ a necessary and sufficient
condition to have $\varphi(G)=\varphi(\mid\hspace{-1mm} G
\hspace{-1mm}\mid)$ is indicated in the following theorem.

\begin{theorem}
For a finite group $G$ we have
$\varphi(G)=\varphi(\mid\hspace{-1mm} G \hspace{-1mm}\mid)$ if and
only if the number of cyclic subgroups of order $\exp(G)$ in $G$
is $\dd\frac{\mid\hspace{-1mm} G \hspace{-1mm}\mid}{{\rm
exp}(G)}$\hspace{0,5mm}.
\end{theorem}

\begin{proof}
Let $n=\hspace{1mm}\mid\hspace{-1mm} G \hspace{-1mm}\mid$, $m={\rm
exp}(G)$ and denote by $k$ the number of cyclic subgroups of order
$m$ in $G$. Then $\varphi(G)=\varphi(m)k$. Since $m \mid n$, we
have $n=mm'$ for some positive integer $m'$. It follows that
$$\varphi(n)=\varphi(mm')=\frac{\gcd(m,m')}{\varphi(
\gcd(m,m'))}\hspace{0,5mm}\varphi(m)\varphi(m'),$$which leads to
$$\varphi(G)=\varphi(\mid\hspace{-1mm} G \hspace{-1mm}\mid) \Longleftrightarrow
\varphi(m)k=\varphi(n) \Longleftrightarrow k=\frac{
\gcd(m,m')}{\varphi( \gcd(m,m'))}\hspace{0,5mm}\varphi(m').$$Now,
a simple arithmetical exercise shows that the last equality above
is equivalent to $k=m'$. Hence
$\varphi(G)=\varphi(\mid\hspace{-1mm} G \hspace{-1mm}\mid)$ if and
only if $k=\dd\frac{\mid G \mid}{\exp(G)}\,.$
\end{proof}

Mention that we were unable to give a precise description of the
finite groups $G$ satisfying the condition in Theorem 3.2.

\bigskip\noindent{\bf Remark.} As we already have seen, there exist large
classes of finite $p$-groups $G$ such that
$\varphi(G)=\varphi(\exp(G))$. Consequently, another natural
problem (similar to b), by replacing $\lvert G \rvert$ with
$\exp(G)$) is to cha\-rac\-te\-ri\-ze the finite groups $G$
satisfying this condition. Under the notation of Theorem 3.2, we
have $\varphi(G)=\varphi(\exp(G))$ if and only if $k=1$, that is
$G$ possesses a unique cyclic subgroup of order $\exp(G)$. In this
case, an interesting remark is given by Theorem 1.1 of \cite{3}:
the group $G$ must be supersolvable. Observe also that for a
finite abelian group $G$ we have
$\varphi(G)=\varphi(\exp(G))\Longleftrightarrow
\varphi(G)=\varphi(\mid\hspace{-1mm} G \hspace{-1mm}\mid)$.
\smallskip

Next we recall an alternative way to define the classical Euler's
totient function, namely
$$\varphi(n)=\hspace{1mm}\mid\hspace{-1mm}
{\rm Aut}(\mathbb{Z}_n)\hspace{-1mm}\mid, \mbox{ for all } n\in
\mathbb{N}^*.$$This leads to the natural idea of comparing the
values $\varphi(G)$ and $\hspace{1mm}\mid\hspace{-1mm} {\rm
Aut}(G)\hspace{-1mm}\mid$ for arbitrary finite groups $G$. First
of all, we observe that for a non-trivial finite group $G$ with
$Z(G)=1$ we have
$$\varphi(G)<\hspace{1mm}\mid\hspace{-1mm} G\hspace{-1mm}\mid\hspace{1mm}=\hspace{1mm}\mid\hspace{-1mm} {\rm
Inn}(G)\hspace{-1mm}\mid\hspace{1mm}\leq\hspace{1mm}\mid\hspace{-1mm}
{\rm Aut}(G)\hspace{-1mm}\mid.$$This inequality is also valid for
many non-cyclic groups of small order, as well as for several
important classes of finite non-abelian groups (e.g. dihedral
groups, hamiltonian groups or finite groups $G$ with
$\varphi(G)=0$). Almost the same thing can be said in the case of
finite abelian groups $G$, for which the study is reduced to
$p$-groups. By using Theorem 2.4 and the explicit formula for
$\mid\hspace{-1mm} {\rm Aut}(G)\hspace{-1mm}\mid$ given by Theorem
4.1 of \cite{7}, we easily obtain the following result.

\begin{theorem}
Let $G$ be a finite abelian group. Then
$\varphi(G)\leq\hspace{1mm}\mid\hspace{-1mm} {\rm
Aut}(G)\hspace{-1mm}\mid$, and we have equality if and only if $G$
is cyclic.
\end{theorem}

Inspired by the above results, we came up with the following
conjecture.

\begin{conjecture}
Let $G$ be a finite group. Then
$\varphi(G)\leq\hspace{1mm}\mid\hspace{-1mm} {\rm
Aut}(G)\hspace{-1mm}\mid$\,, and we have equality if and only if
$G$ is cyclic.
\end{conjecture}

We were sure that the above conjecture is true for a long time,
but finally we disproved it. Several interesting remarks are
presented in the following.

\bigskip\noindent{\bf Remarks.}
1. A well-known conjecture in the theory of finite groups is the
so-called "LA-conjecture", which asserts that for a finite
non-cyclic $p$-group $G$ of order greater than $p^2$,
$\mid\hspace{-1mm}G\hspace{-1mm}\mid$ divides $\mid\hspace{-1mm}
{\rm Aut}(G)\hspace{-1mm}\mid$\,. It follows that
$$\varphi(G)<\hspace{1mm}\mid\hspace{-1mm}
G\hspace{-1mm}\mid\hspace{1mm}\leq\hspace{1mm}\mid\hspace{-1mm}
{\rm Aut}(G)\hspace{-1mm}\mid.$$Since the inequality
$\varphi(G)\leq\hspace{1mm}\mid\hspace{-1mm} {\rm
Aut}(G)\hspace{-1mm}\mid$ also holds for groups of order $p$ or
$p^2$, we infer that it is true for arbitrary $p$-groups.
Obviously, this remark can be extended to nilpotent groups. In
other words, \textit{Conjecture {\rm 3.4} is true for all finite
nilpotent groups}.
\smallskip

2. Assume the finite group $G$ of order
$n=p_1^{\alpha_1}p_2^{\alpha_2}\cdots p_k^{\alpha_k}$ and exponent
$m=p_1^{\beta_1}p_2^{\beta_2}\cdots p_k^{\beta_k}$ to be a
counterexample for Conjecture 3.4, that is
$\varphi(G)>\hspace{1mm}\mid\hspace{-1mm} {\rm
Aut}(G)\hspace{-1mm}\mid$\,. Then the map $$f:M(G)=\{a\in G\mid
o(a)=m\}\longrightarrow {\rm Inn}(G)$$
$$\hspace{19mm}a\mapsto f_a, \mbox{ the inner automorphism induced by }a,$$is not
one-to-one and therefore there are $a, b\in M(G)$ such that
$c=ab\in Z(G)$. Since $ab=ba$, we infer that $G$ contains a
subgroup isomorphic to $\mathbb{Z}_m\times\mathbb{Z}_m$, namely
$\langle a, b \rangle$, and that $Z(G)$ contains a cyclic subgroup
of order $m$, namely $\langle c \rangle$. Consequently:
\begin{description}
\item[\hspace{2mm}{\rm (i)}] $n\geq m^2$.
\item[\hspace{1mm}{\rm (ii)}] Every Sylow $p_i$-subgroup $S_i$ of $G$ has a subgroup $H_i\cong\mathbb{Z}_{p_i^{\beta_i}}\times\mathbb{Z}_{p_i^{\beta_i}}$. In particular, $S_i$ is not cyclic and $\alpha_i\geq 2\beta_i\geq 2$.
\item[{\rm (iii)}] $Z(S_i)$ has a cyclic subgroup of order $p_i^{\beta_i}={\rm exp}(S_i)$, for all $i=\ov{1,k}$.
\end{description}
Thus, under the above notation, we deduce that \textit{Conjecture
{\rm 3.4} is true for all finite groups G satisfying}
$\mid\hspace{-1mm}G\hspace{-1mm}\mid\hspace{1mm}<\exp(G)^2$
\textit{or possessing a cyclic Sylow subgroup}.
\smallskip

3. In order to give a counterexample for Conjecture 3.4, we must
look at the finite groups with "few" automorphisms. Examples of
such groups can be constructed by taking direct products of type
$G=G_1\times G_2$, where $G_1$ is a cyclic group and $G_2$ is a
group with trivial center (usually, a simple group -- see the
technique developed in \cite{2}). Then the number of automorphisms
of $G$ can be easily computed, according to the main theorem of
\cite{1}:
$$\mid\hspace{-1mm} {\rm Aut}(G)\hspace{-1mm}\mid\hspace{1mm}=\hspace{1mm}\mid\hspace{-1mm} {\rm Aut}(G_1)\hspace{-1mm}\mid\mid\hspace{-1mm} {\rm Aut}(G_2)\hspace{-1mm}\mid\mid\hspace{-1mm} {\rm Hom}(G_2,
G_1)\hspace{-1mm}\mid.$$Two significant examples are the
following. \bigskip

{\bf\hspace{5mm} Example 3.1.} Put $G_1=\mathbb{Z}_6$ and
$G_2=S_3$. One obtains
$$\mid\hspace{-1mm} {\rm Aut}(\mathbb{Z}_6\times S_3)\hspace{-1mm}\mid\hspace{1mm}=\hspace{1mm}\mid\hspace{-1mm} {\rm Aut}(\mathbb{Z}_6)\hspace{-1mm}\mid\mid\hspace{-1mm} {\rm Aut}(S_3)\hspace{-1mm}\mid\mid\hspace{-1mm} {\rm Hom}(S_3,
\mathbb{Z}_6)\hspace{-1mm}\mid\hspace{1mm}=\hspace{1mm}2\cdot6\cdot2=24.$$On
the other hand, we have
$$\varphi(\mathbb{Z}_6\times S_3)=20$$and therefore even though $\mathbb{Z}_6\times S_3$
sa\-tis\-fies all above conditions (i)-(iii), it fails to give a
co\-un\-ter\-exam\-ple for Conjecture 3.4. Observe also that
$\mathbb{Z}_6\times S_3$ constitutes an example of a finite group
whose order is greater than the order of its automorphism group.
\bigskip

{\bf\hspace{5mm} Example 3.2.} Put $G_1=\mathbb{Z}_m$ and
$G_2=M_{11}$, where $m={\rm exp}(M_{11})$. Then $330\mid m$, since
$\mid\hspace{-1mm}M_{11}\hspace{-1mm}\mid\hspace{1mm}=2^4\cdot3^2\cdot5\cdot11$.
By the remark on page 382 of \cite{2}, one obtains
$$\mid\hspace{-1mm} {\rm Aut}(\mathbb{Z}_m\times M_{11})\hspace{-1mm}\mid\hspace{1mm}=\hspace{1mm}\mid\hspace{-1mm} {\rm
Aut}(\mathbb{Z}_m)\hspace{-1mm}\mid\mid\hspace{-1mm}M_{11}\hspace{-1mm}\mid\hspace{1mm}=\hspace{1mm}\varphi(m)\mid\hspace{-1mm}M_{11}\hspace{-1mm}\mid.$$On
the other hand, it is clear that ${\rm exp}(\mathbb{Z}_m\times
M_{11})=m$ and that the number of elements of order $m$ in
$\mathbb{Z}_m\times M_{11}$ is greater than
$\varphi(m)\mid\hspace{-1mm}M_{11}\hspace{-1mm}\mid$ (the set
$M(\mathbb{Z}_m\times M_{11})$ contains all elements $(a,b)$ with
$o(a)=m$ and $b$ arbitrary, as well as all elements $(a,b)$ with
$o(a)=m/5$ and $o(b)=5$, or $o(a)=m/11$ and $o(b)=11$. These show
that
$$\varphi(\mathbb{Z}_m\times M_{11})>\hspace{1mm}\mid\hspace{-1mm} {\rm
Aut}(\mathbb{Z}_m\times M_{11})\hspace{-1mm}\mid$$and hence
\textit{Conjecture {\rm 3.4} is not true for all finite groups}.
\bigskip

We end this section by mentioning that we were not able to answer
the following question: \textit{Are the finite cyclic groups the
unique groups G satisfying}
$\varphi(G)=\hspace{1mm}\mid\hspace{-1mm} {\rm
Aut}(G)\hspace{-1mm}\mid$\,?

\section{Finite groups $G$ for which $\varphi(G)\neq 0$}

In this section we will denote by $\calc$ the class of finite groups
$G$ satisfying $\varphi(G)\neq 0$. An immediate characterization
of $\calc$ is given by the following theorem.

\begin{theorem}
A finite group $G$ of order $n$ belongs to $\calc$ if and only if
its set of element orders $\pi_e(G)$ forms a sublattice of the
lattice of all divisors of $n$.
\end{theorem}

We know that $\calc$ contains some important classes of groups, as
the finite nilpotent groups or the dihedral groups $D_{2n}$ with
$n$ even (notice that another interesting example of a
non-nilpotent finite group contained in $\calc$ has been given by
Professor Derek Holt on MathOverflow -- see \cite{16}). We also
have seen that the symmetric groups $S_n$, for $n\geq 3$, and the
alternating groups $A_n$, for $n\geq 4$, do not belong to
$\calc$. These examples allow us to infer some elementary
properties of $\calc$:
\bigskip

Since $D_{12}\cong\mathbb{Z}_2\times S_3$ is contained in $\calc$,
but it possesses a subgroup, as well as a quotient, isomorphic to
$S_3$, it follows that $\calc$ is not closed under subgroups or
homomorphic images. On the other hand, $S_3$ is an extension of
two groups in $\calc$, namely $\mathbb{Z}_2$ and $\mathbb{Z}_3$,
and its subgroup lattice $L(S_3)$ is isomorphic to
$L(\mathbb{Z}_3\times \mathbb{Z}_3)$. These show that $\calc$ is
also not closed under extensions or $L$-isomorphisms. On the other
hand, $\calc$ is obviously closed under direct products.
\bigskip

Another class of finite groups for which we are able to
characterize the containment to $\calc$ is constituted by {\it
metacyclic groups}. It is well-known that such a group $G$ has a
presentation of the form
\begin{equation}
    \langle a, b \mid a^m = 1, b^n = a^s, b^{-1} a b = a^r \rangle ,\tag{4}
\end{equation}
where $\gcd(m,r)=1$ and $r^n\equiv 1 \hspace{1mm}({\rm
mod} \hspace{1mm}m)$. In particular, if $s=0$, then $G$ is called
{\it split metacyclic}.

\begin{theorem}
A metacyclic group $G$ with the above presentation is contained in
$\calc$ if and only if $n \mid \gcd(m,s)$. In particular, if $G$
is split metacyclic, then it is contained in $\calc$ if and only
if $n \mid m$.
\end{theorem}

\begin{proof}
Suppose first that $G$ is split.
Then we can easily compute the powers of elements of $G$, namely
$$(b^ia^j)^k=b^{ik}a^{j\frac{r^{ik}-1}{r^i-1}}, \mbox{ for all } 1\leq i\leq
n \mbox{ and } 1\leq j\leq m.$$It follows that $\exp(G)=
\lcm(m,n)$ and therefore $G$ belongs to $\calc$ if and only if
$n \mid m$.

If $G$ is an arbitrary metacyclic group given by (4), then the
order of $b$ is
$$n_1=\frac{mn}{\gcd(m,s)}\hspace{1mm}.$$In other words, we have
$$G= \langle a, b \mid a^m = 1, b^{n_1} = 1, b^{-1} a b = a^r \rangle .$$By
applying the first part of our proof, one obtains
$$\exp(G)=\lcm(m,n_1)=\lcm(\frac{m}{\gcd(m,s)}\hspace{1mm}
\gcd(m,s),\frac{m}{\gcd(m,s)}\hspace{1mm}n)$$
$$\hspace{-28mm}=\frac{m}{\gcd(m,s)}\hspace{1mm}\lcm(
\gcd(m,s),n).$$So, $G$ belongs to $\calc$ if and only if $n_1
\mid m$, that is $n \mid \gcd(m,s)$.
\end{proof}

\noindent{\bf Remark.} By taking $n=2$ and $r=m-1$ in Theorem 4.2,
$G$ becomes the dihedral group $D_{2m}$. In this way, $D_{2m}$
belongs to $\calc$ if and only if $m$ is even, a result that can
be also inferred directly from Theorem 2.6.

\begin{theorem}
If $G$ is a finite group of exponent $m$, then the direct product
$\mathbb{Z}_m\times G$ is contained in $\calc$. In particular, any
finite group is a quotient of a group in $\calc$.
\end{theorem}

\begin{proof}
Since $o(a,b)=\lcm(o(a),o(b))$, for all $(a,b)\in
\mathbb{Z}_m\times G$, one obtains that $\exp(\mathbb{Z}_m\times
G)=m$. On the other hand, $\mathbb{Z}_m\times G$ obviously
possesses elements of order $m$ (for example, any element of type
$(b,1)$, where $b$ is a generator of $\mathbb{Z}_m$).
\end{proof}

\begin{cor}
$\calc$ is not contained in the class of finite solvable groups.
\end{cor}

\begin{proof}
Let $G$ be a non-solvable group with $\exp(G)=m$. Then
$\mathbb{Z}_m\times G$ is contained in $\calc$ by the above
theorem. On the other hand, $\mathbb{Z}_m\times G$ is not
sol\-va\-ble since the class of (finite) solvable groups is closed
under homomorphic images.
\end{proof}

\noindent{\bf Remark.} The converse inclusion also fails (in fact,
even the classes of CLT-groups or supersolvable groups are not
contained in $\calc$). We are also able to describe the
intersections of $\calc$ with the classes of ZM-groups (i.e. the
finite groups with all Sylow subgroups cyclic) and CP-groups (i.e.
the finite groups with all elements of prime power orders): they
consist of the finite cyclic groups and of the finite $p$-groups,
respectively.
\bigskip

In order to study whether a finite group $G$ of
exponent $\prod_{i=1}^k p_i^{\beta_i}$ belongs to $\calc$, the connections
between the sets $M(G)=\{a\in G \mid o(a)=\exp(G)\}$ and
$M_i(G)=\{a\in G \mid o(a)=p_i^{\beta_i}\}$, $i=1,2,\dots,k$, are
essential. Clearly, if $G$ is nilpotent, then there is a bijection
from $M(G)$ to the cartesian product of $M_i(G)$, $i=1,2,...,k$.
Since every set $M_i(G)$ is nonempty, so is $M(G)$ and therefore
$G$ belongs to $\calc$. In general, we have
\begin{equation}
M(G)\neq \emptyset \Longleftrightarrow \exists \hspace{1mm}a_i\in M_i(G),
    i=\ov{1,k}, \text{with} \hspace{1mm} a_ia_j=a_ja_i
    \hspace{1mm}\text{for all} \hspace{1mm} i\neq j\,. \tag{5}
\end{equation}
We remark that the condition in the right side of (5) can be
obtained (by replacing $"\forall"$ with $"\exists"$) from the
condition
$$\forall \hspace{1mm}a_i\in M_i(G), i=\ov{1,k}, \text{we have} \hspace{1mm} a_ia_j=a_ja_i
\hspace{1mm}\text{for all} \hspace{1mm} i\neq j\,,$$that
characterizes the nilpotency of $G$.
\bigskip

Obviously, a finite group $G$ is contained in $\calc$ if and only
if $\varphi(G)=p$\, for some non-zero positive integer $p$. In
this way, by fixing $p\in\mathbb{N}^*$, the study of the equation
$\varphi(G)=p$ (where the solutions $G$ are considered up to
group isomorphism) is essential. We end this section by solving it
in the particular case when $p$ is a prime.

As we have seen in Section 2, if ${\rm exp}(G)=m$, then
$\varphi(G)=\varphi(m)k$, where $k$ denotes the number of cyclic
subgroups of order $m$ in $G$. On the other hand, we already know
that for $p$ odd our equation has a solution if and only if $p$ is
of type $2^q-1$, and moreover this solution is unique: the
elementary abelian group $\mathbb{Z}_2^q$. So, in the following we
may assume that $p=2$. This implies $k=1$ and $\varphi(m)=2$, that
is $m\in\{3,4,6\}$.

\medskip{\bf\hspace{10mm} Case 1.} $m=3$

\noindent In this case $G$ is an elementary abelian 3-group. Since
it possesses only one subgroup of order 3, one obtains $G\cong
\mathbb{Z}_3$.

\medskip{\bf\hspace{10mm} Case 2.} $m=4$

\noindent In this case $G$ is a 2-group. Let $\langle x\rangle$ be
the unique cyclic subgroup of order 4 of $G$. Since the quotient
$G/\langle x^2\rangle$ is elementary abelian, we infer that
$G'\subseteq\Phi(G)\subseteq\langle x^2\rangle$. If $G'$ is
trivial, then $G$ is abelian. Clearly, the uniqueness of $\langle
x\rangle$ implies that $G\cong\mathbb{Z}_4$. Suppose now that
$G'=\Phi(G)=\langle x^2\rangle$ and let $y\in Z(G)$. Then the
cyclic subgroup $\langle xy\rangle$ is of order 4. Therefore we
have $\langle xy\rangle=\langle x\rangle$, that is $y\in\langle
x\rangle$. This shows that $Z(G)\subseteq\langle x\rangle$. If
$Z(G)=\langle x\rangle$, then we easily obtain that all elements
of $G$ are contained in $Z(G)$, i.e. $G$ is abelian, a
contradiction. In this way, $Z(G)$ also coincides with $\langle
x^2\rangle$, proving that $G$ is extraspecial. It follows that $G$
is either a central product of $n$ dihedral groups of order 8 or a
central product of $n-1$ dihedral groups of order 8 and a
quaternion group of order 8, where $\mid
G\mid\hspace{1mm}=2^{2n+1}$ (see Theorem 4.18 of \cite{14}, II).
By using again the uniqueness of $\langle x\rangle$, we infer that
$n=1$ and $G\cong D_8$.

\medskip{\bf\hspace{10mm} Case 3.} $m=6$

\noindent Let $H$ be the unique cyclic subgroup of order 6 in $G$
and $M_1\cong\mathbb{Z}_2$, $M_2\cong\mathbb{Z}_3$ be the (unique)
non-trivial subgroups of $H$. Since $H$ is normal, one obtains
$H=H^x=(M_1M_2)^x=M_1^xM_2^x$, for all $x\in G$, which implies
that both $M_1$ and $M_2$ are also normal in $G$.

Let $M_2'\leq G$ with $\mid M_2'\mid\hspace{1mm}=3$. Then
$M_1M_2'$ is a subgroup of order 6 of $G$ and $M_1$ is normal in
$M_1M_2'$. We infer that $M_1M_2'\cong\mathbb{Z}_6$ and so
$M_1M_2'=H$. This leads to $M_2'=M_2$, that is $M_2$ is the unique
subgroup of order 3 of $G$. It is well-known that a $p$-group
containing only one subgroup of order $p$ is either cyclic or
generalized quaternion (see, for example, (4.4) of \cite{14}, II).
In our case, it follows that all Sylow 3-subgroups of $G$ are
cyclic. Then $M_2$ is in fact the unique Sylow 3-subgroup of $G$,
because ${\rm exp}(G)=6$. This leads to $\mid
G\mid\hspace{1mm}=2^n \cdot 3$ for some $n\in\mathbb{N}^*$.

Let $n_2$ the number of Sylow 2-subgroups of $G$ and denote by $S$
such a subgroup. Then $S$ is elementary abelian and we have
$$n_2\mid3 \mbox{ and } n_2\equiv 1 \hspace{1mm}({\rm mod}\hspace{1mm}2),$$
i.e. $n_2\in\{1,3\}$. If $n_2=1$, then $G$ is nilpotent, more
precisely $G\cong S\times M_2$. It is now easy to see that the
uniqueness of a cyclic subgroup of order 6 in $G$ implies that $S$
is of order 2 (in fact $S=M_1$) and thus $G\cong\mathbb{Z}_6$. If
$n_2=3$, then we must have $\mid {\rm
Core}_G(S)\mid\hspace{1mm}=2$. Since $G/{\rm Core}_G(S)$ can be
embedded in $S_3$, we infer that $\mid G\mid\hspace{1mm}=12$ and
$G\cong D_{12}$.
\bigskip

Hence we have proved the following theorem.

\begin{theorem}
Let $p$ be a prime. Then the equation $\varphi(G)=p$ has solutions
if and only if either $p$ is odd of the form $2^q-1$ or $p=2$. In
the first case we have a unique solution, namely the elementary
abelian {\rm 2}-group $\mathbb{Z}_2^q$, while in the second case
we have five non-isomorphic solutions, namely the cyclic groups
$\mathbb{Z}_3$, $\mathbb{Z}_4$, $\mathbb{Z}_6$ and the dihedral
groups $D_8$ and $D_{12}$.
\end{theorem}

\section{Conclusions and further research}

The study of the structure of finite groups through their sets of
elements of maximal orders is an interesting and difficult topic
in finite group theory. In our paper we introduced a
generalization of the Euler's totient function that counts the
elements of (maximal) order $\exp(G)$ of a finite group $G$. It is
clear that the study of its properties and applications, as well
as the study of several related classes of finite groups (as
$\calc$), can be continued in many ways. These will surely be the
subject of some further research.
\bigskip

We end this paper by indicating a list of open problems concerning
our previous results.
\bigskip

\noindent{\bf Problem 5.1.} Answer the question in the end of
Section 3.
\bigskip

\noindent{\bf Problem 5.2.} Compute explicitly $\varphi(G)$ for
other remarkable classes of finite groups $G$, such as ZM-groups,
metacyclic groups, supersolvable groups, solvable groups, \dots,
and so on.
\bigskip

\noindent{\bf Problem 5.3.} Study the equation $\varphi(G)=p$,
where $p$ is an {\it arbitrary} positive integer. When does this
equation have a unique solution (up to group isomorphism)?
\bigskip

\noindent{\bf Problem 5.4.} Given a finite group $G$, in general
we don't have $\varphi(H)\mid\varphi(G)$\newline or
$\varphi(H)\leq\varphi(G)$ for all subgroups $H$ of $G$. Study the
classes $\calc_1$ and $\calc_2$ consisting of all finite groups
that satisfy these conditions (remark that both $\calc_1$ and
$\calc_2$ contain the finite cyclic groups, but they also contain
some noncyclic groups, such as $Q_8$).
\bigskip

\noindent{\bf Problem 5.5.} Study the classes $\calc^*$ and
$\calc^{**}$ of all finite groups all of whose
subgroups (respectively quotients) belong to $\calc$ (remark that a group
$G$ in $\calc^*$ satisfies a well-known condition in finite group
theory, usually called the "$pq$-condition": for any two distinct
primes $p$ and $q$, every subgroup of order $pq$ of $G$ is
cyclic).
\bigskip
\bigskip

\noindent{\bf Acknowledgements}. The author wishes to thank
Professor Tom De Medts for his help in proving Theorem 2.8.

\vspace*{5ex}\small

\hfill
\begin{minipage}[t]{5cm}
Marius T\u arn\u auceanu \\
Faculty of  Mathematics \\
``Al.I. Cuza'' University \\
Ia\c si, Romania \\
e-mail: {\tt tarnauc@uaic.ro}
\end{minipage}

\end{document}